\documentclass[11pt,reqno]{amsart}

\usepackage[foot]{amsaddr}
\usepackage{amsmath,amssymb,amsthm}
\usepackage{hyperref}
\usepackage{booktabs}
\usepackage{multirow}
\usepackage{enumerate}
\usepackage{bm}
\usepackage{tikz}
\usetikzlibrary{arrows,shapes}
\tikzset{>=stealth}

\usepackage[maxbibnames=99,
style=numeric,
backend=bibtex,
sorting=nyt,
sortcites,
doi = false,
isbn = false,
url = false,
natbib=true]{biblatex}

\theoremstyle{plain}
\newtheorem{lemma}{Lemma}

\newtheorem{proposition}{Proposition}
\newtheorem{corollary}{Corollary}

\theoremstyle{definition}

\newtheorem{example}{Example}
\newtheorem{problem}{Problem}

\theoremstyle{remark}

\newcommand{\reals}{\mathbb{R}}
\newcommand{\ints}{\mathbb{Z}}

\newcommand{\vect}[1]{\bm{#1}}

\let\leq\leqslant
\let\geq\geqslant
\let\eps\varepsilon

\newcommand{\Aout}{A_{\text{out}}}
\newcommand{\Ain}{A_{\text{in}}}
\newcommand{\Aoff}{A_{\text{off}}}
\newcommand{\Aon}{A_{\text{on}}}
\newcommand{\Z}{Z(n,\vect\alpha,\vect\beta,\vect\gamma,\vect\delta)}

\DeclareMathOperator{\conv}{conv}
\DeclareMathOperator{\proj}{proj}

\title{Tight MIP formulations for bounded length cyclic sequences}

\author[Thomas Kalinowski]{Thomas Kalinowski$^{1,3}$}
\address{$^1$School of Science and Technology, University of New England, Armidale, NSW 2351, Australia}
\author[Tomas Lid\'en]{Tomas Lid\'en$^2$}
\address{$^2$Department of Science and Technology, Link\"oping University, Norrk\"oping SE-601 74, Sweden}
\author[Hamish Waterer]{Hamish Waterer$^3$}
\address{$^3$School of Mathematical \& Physical Sciences, University of Newcastle, NSW 2308, Australia}

\email{tomas.liden@liu.se}
\email{tkalinow@une.edu.au}
\email{hamish.waterer@newcastle.edu.au}

\date{}

\keywords{Production sequencing; Bounded up/down times; Extended formulations; Convex hulls}
\subjclass[2010]{90C11, 90C27, 90C35, 90C57}

\begin{document}

\begin{abstract}
  We study cyclic binary strings with bounds on the lengths of the intervals of consecutive ones and
  zeros. This is motivated by scheduling problems where such binary strings can be used to represent
  the state (on/off) of a machine. In this context the bounds correspond to minimum and maximum
  lengths of on- or off-intervals, and cyclic strings can be used to model periodic
  schedules. Extending results for non-cyclic strings is not straight forward. We present a
  non-trivial tight compact extended network flow formulation, as well as valid inequalities in the
  space of the state and start-up variables some of which are shown to be facet-defining. Applying a
  result from disjunctive programming, we also convert the extended network flow formulation into an
  extended formulation over the space of the state and start-up variables.
\end{abstract}

\maketitle

\section{Introduction}
In scheduling problems it is often natural to use time-indexed
binary variables to model the availability of resources, such as the state of machines (on/off) or
roster patterns for the workforce. In these contexts there are often bounds on the lengths of on- and
off-intervals, and there is a significant literature on mixed integer programming formulations for
this~\cite{Frangioni2006,Frangioni2009,Gentile2017,Hedman2009,Lee2004,Malkin2003,QueyranneWolsey2017}. In
particular, \citet{Malkin2003} showed that for lower bounds on the lengths of on- and off-intervals,
the valid inequalities that can be found in~\cite{Wolsey1998} are sufficient to describe the convex
hull in the space of the state and start-up variables. \citet{PochetWolsey2006} give the convex hull
for the case of constant upper and lower bounds, and this was generalized
by~\citet{QueyranneWolsey2017} who considered upper and lower bounds, and allowed these bounds to
vary over time. They present a tight extended network formulation, and obtain the convex hull in
the space of the state and start-up variables via a projection from a different path formulation.

Our work is motivated by applications in the scheduling of railway maintenance~\cite{Liden2015},
where it is required in some situations that schedules are cyclic. For this reason, we let the
sequence of state variables ``wrap around'' the time horizon and apply the bounds on the lengths of
on- or off-intervals also to intervals that start in the end of the time horizon and continue in the
beginning. A more formal problem description is provided in Section~\ref{sec:problem}. In
Section~\ref{sec:flow}, we follow the approach from~\cite{QueyranneWolsey2017} to derive a compact
extended network flow formulation. It turns out that the straightforward cyclic variant of the
network formulation from\cite{QueyranneWolsey2017} does not lead to an integral polytope in the
space of the flow variables, but we can obtain an integral network flow formulation by considering a
larger network that arises from exploiting a simple disjunction. In Section~\ref{sec:yz-formulation}
we study a cyclic variant of the Queyranne/Wolsey formulation in the space of the state and start-up
variables. We prove that it is a valid formulation, but in contrast to the non-cyclic case the
polytope is not integral. For the case that the bounds on the interval lengths are constant over
time we provide some valid inequalities, and give sufficient conditions for them to be
facet-defining. We also use a result from disjunctive programming to derive an extended formulation
for the convex hull in the space of the state and start-up variables.  Finally, in
Section~\ref{sec:conclusion} we describe some directions for further investigations.

\section{Problem description}\label{sec:problem}
Throughout this paper, we denote the set $\{a,a+1,\dots,b\}$ for integers $a\leq b$ by $[a,b]$. Let
the time horizon be indexed by $[0,n-1]$ with the convention that time is added modulo $n$, that is,
$0$ is the time period after $n-1$. For integers $a$ and $b$ with $0\leq b<a<n$ representing time
periods we let the interval wrap around in the natural way, that is,
$[a,b]=\{a,a+1,\dots,n-1,0,1,\dotsc,b\}$.

As in~\cite{QueyranneWolsey2017}, we consider parameters $(\vect\alpha,\,\vect\beta,\,\vect\gamma,\,\vect\delta)\in\ints^{4n}$ that impose bounds on the length of on- and off-intervals in the following way:
\begin{itemize}
\item $\alpha_t\in[1,n-1]$ is a lower bound on the length of an on-interval starting in period $t$,
\item $\beta_t\in[\alpha_t,n-1]$ is an upper bound on the length of an on-interval starting in period
  $t$,
\item $\gamma_t\in[1,n-1]$ is a lower bound on the length of an off-interval starting in period $t$,
\item $\delta_t\in[\gamma_t,n-1]$ is an upper bound on the length of an off-interval starting in period $t$.
\end{itemize}
In particular, we require that there are at least one on-period and at least one off-period
(otherwise there is an on-interval of length $n$ or an off-interval of length $n$, and no matter
where we let this start the upper bound on the length of the corresponding interval will be
violated). We define binary state variables $y_t$ for $t\in[0,n-1]$ to be
\[y_t=
  \begin{cases}
    1 &\text{if period $t$ is an on-period},\\
    0 &\text{if period $t$ is an off-period}.
  \end{cases}
\]
The set of feasible state sequences $(y_0,\dotsc,y_{n-1})\in\{0,1\}^n$ is characterized by the following implications:
\begin{align}
  y_t-y_{t-1}=1 &\implies y_{t+i}=1 \text{ for all }i\in[0,\alpha_t-1] &&t\in[0,n-1],\label{eq:lower_bound_on}\\
  y_t-y_{t-1}=1 &\implies y_{t+i}=0 \text{ for some }i\in[\alpha_t,\beta_t]
                                           &&t\in[0,n-1],\label{eq:upper_bound_on}\\
  y_{t-1}-y_t=1 &\implies y_{t+i}=0 \text{ for all }i\in[0,\gamma_t-1] &&t\in[0,n-1],\label{eq:lower_bound_off}\\
  y_{t-1}-y_t=1 &\implies y_{t+i}=1 \text{ for some }i\in[\gamma_t,\delta_t] &&t\in[0,n-1].\label{eq:upper_bound_off}
\end{align}
We define the binary start-up variables $z_t$ for $t\in[0,n-1]$ to be
\begin{align}
  z_t=1 &\iff y_{t-1}=0\text{ and }y_{t}=1, \label{eq:switch_on}
\end{align}
and define the set
\[\Z=\left\{(\vect y,\vect
    z)\in\{0,1\}^{2n}\,:\,~(\ref{eq:lower_bound_on}),~(\ref{eq:upper_bound_on}),~(\ref{eq:lower_bound_off}),~(\ref{eq:upper_bound_off}),~(\ref{eq:switch_on}),\
    1\leq y_0+\dotsb+y_{n-1}\leq n-1\right\}.\] We are interested in tight linear formulations for
$\Z$, and our approach is to adapt the arguments used in~\cite{QueyranneWolsey2017}. Before studying
the general case we derive a simple feasibility criterion for the constant bound case in which the bounds $\eps_t$ for
$\vect\eps\in\{\vect\alpha,\vect\beta,\vect\gamma,\vect\delta\}$ do not change over time. If
$(\alpha_t,\beta_t,\gamma_t,\delta_t)=(\alpha,\beta,\gamma,\delta)$ for all $t$ then
the number of start-up periods is an integer between $n/(\beta+\delta)$ and $n/(\alpha+\gamma)$. The
following proposition states that for every integer $k$ in this range there exists a feasible
solution with $k$ start-up periods.
\begin{proposition}
  If $(\alpha_t,\,\beta_t,\,\gamma_t,\,\delta_t)=(\alpha,\,\beta,\,\gamma,\,\delta)$ all $t\in[0,n-1]$, then
  \[\left\{z_0+\dots+z_{n-1}\,:\,(\vect y,\vect z)\in \Z\right\}=\left\{
    k\in\mathbb{Z}\,:\,n/(\beta+\delta)\leqslant k\leqslant n/(
    \alpha+\gamma)\right\}.\]
In particular, $\Z\neq\emptyset$ if and only if
$k\left(\alpha+\gamma\right)\leq n\leq k\left(\beta+\delta\right)$ for some integer $k$.
\end{proposition}
\begin{proof}
  Let
  $K=\left\{ k\in\mathbb{Z}:n/\left(\beta+\delta\right)\leqslant k\leqslant n/\left(
      \alpha+\gamma\right)\right\}$. We have to show that there exists $(\vect y,\vect z)\in \Z$
  with $z_0+\dotsb+z_{n-1}=k$ if and only if $k\in K$.  First, suppose $(\vect y,\vect z)\in \Z$,
  set $k=z_0+\dotsb+z_{T-1}$ and let $0\leq t_1<t_2<\dots<t_k\leq n-1$ denote the indices with
  $z_{t_i}=1$ for $i\in[1,k]$. Then, for every $i\in[1,k]$, $t_{i+1}=t_i+p_i+q_i$ with
  $\alpha\leq p_i\leq\beta$ and $\gamma\leq q_i\leq\delta$ for all $i\in[1,k]$. Summing over $i$, we
  obtain $n=(p_1+q_1)+\dots+(p_k+q_k)$, hence $k(\alpha+\gamma)\leq n\leq k(\beta+\delta)$, which
  implies $k\in K$. For the converse, start with any $k\in K$. Then
  $k(\alpha+\gamma)\leq n\leq k(\beta+\delta)$, hence $n-(\alpha+\gamma)\geq (k-1)(\alpha+\gamma)$
  and $n-(\beta+\delta)\leq (k-1)(\beta+\delta)$. This implies that we can choose
  $p_k\in[\alpha,\beta]$ and $q_k\in[\gamma,\delta]$ such that
  $(k-1)(\alpha+\gamma)\leq n-(p_k+q_k)\leq(k-1)(\beta+\delta)$.  Continuing this way, we obtain
  $n=(p_k+q_k)+(p_{k-1}+q_{k-1})+\dotsb+(p_1+q_1)$ with $p_i\in[\alpha,\beta]$ and
  $q_i\in[\gamma,\delta]$. Then
 \begin{align*}
    \vect y &=
              \underbrace{11\dots1}_{p_1}\underbrace{00\dots0}_{q_1}\underbrace{11\dots1}_{p_2}\underbrace{00\dots0}_{q_2}\
              \dots\ \underbrace{11\dots1}_{p_k}\underbrace{00\dots0}_{q_k},&   
    \vect z &=
              1\underbrace{00\dots0}_{p_1+q_1-1}1\underbrace{00\dots0}_{p_2+q_2-1}\ \dots\
              1\underbrace{00\dots0}_{p_k+q_k-1}
  \end{align*}
defines a vector $(\vect y,\vect z)\in\Z$ satisfying $z_1+\dots+z_{n-1}=k$.
\end{proof}

\section{An extended network formulation}\label{sec:flow}
We consider a directed graph $(V,A)$ with node set $V=\{0,1\}\times[0,n-1]$, and arc set
\begin{align*}
  A &= \{((0,t),(1,l))\ :\ l\in[t+\alpha_t,\,t+\beta_t]\}\cup\{((1,t),(0,l))\ :\ l\in[t+\gamma_t,\,t+\delta_t]\}.
\end{align*}
Figure~\ref{fig:network} illustrates this graph for $n=6$ and
$(\alpha_t,\beta_t,\gamma_t,\delta_t)=(1,2,1,2)$ for all $t\in[0,n-1]$. 
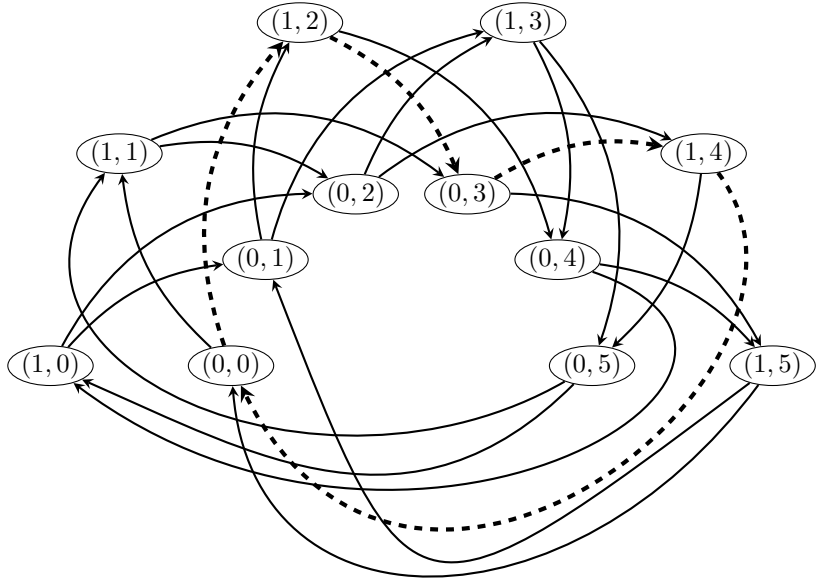
\begin{figure}[htb]
  \centering
  \begin{tikzpicture}[scale=1.2]
    \draw[use as bounding box,color=white](-4.8,-2.5)rectangle(4.8,4.3);
    \node[ellipse,inner sep=.3pt,outer sep=.2pt,draw] (v0) at (180:2) {{\small $(0,0)$}};
    \node[ellipse,inner sep=.3pt,outer sep=.2pt,draw] (v1) at (144:2) {{\small $(0,1)$}};
    \node[ellipse,inner sep=.3pt,outer sep=.2pt,draw] (v2) at (108:2) {{\small $(0,2)$}};
    \node[ellipse,inner sep=.3pt,outer sep=.2pt,draw] (v3) at (72:2) {{\small $(0,3)$}};
    \node[ellipse,inner sep=.3pt,outer sep=.2pt,draw] (v4) at (36:2) {{\small $(0,4)$}};
    \node[ellipse,inner sep=.3pt,outer sep=.2pt,draw] (v5) at (0:2) {{\small $(0,5)$}};
    \node[ellipse,inner sep=.3pt,outer sep=.2pt,draw] (w0) at (180:4) {{\small $(1,0)$}};
    \node[ellipse,inner sep=.3pt,outer sep=.2pt,draw] (w1) at (144:4) {{\small $(1,1)$}};
    \node[ellipse,inner sep=.3pt,outer sep=.2pt,draw] (w2) at (108:4) {{\small $(1,2)$}};
    \node[ellipse,inner sep=.3pt,outer sep=.2pt,draw] (w3) at (72:4) {{\small $(1,3)$}};
    \node[ellipse,inner sep=.3pt,outer sep=.2pt,draw] (w4) at (36:4) {{\small $(1,4)$}};
    \node[ellipse,inner sep=.3pt,outer sep=.2pt,draw] (w5) at (0:4) {{\small $(1,5)$}};
    \draw[thick,bend right=-20,->] (v0) to (w1);
    \draw[ultra thick,dashed,bend right=-30,->] (v0) to (w2);
    \draw[thick,bend right=-20,->] (v1) to (w2);
    \draw[thick,bend right=-30,->] (v1) to (w3);
    \draw[thick,bend right=-20,->] (v2) to (w3);
    \draw[thick,bend right=-30,->] (v2) to (w4);
    \draw[ultra thick,dashed,bend right=-20,->] (v3) to (w4);
    \draw[thick,bend right=-30,->] (v3) to (w5);
    \draw[thick,bend right=-20,->] (v4) to (w5);
    \draw[thick,->] (v4) .. controls (0:5) and (270:3) .. (w0);
    \draw[thick,->] (v5) .. controls (270:2) and (210:2) .. (w0);
    \draw[thick,->] (v5) .. controls (240:2) and (180:5) .. (w1);
    \draw[thick,bend right=-20,->] (w0) to (v1);
    \draw[thick,bend right=-30,->] (w0) to (v2);
    \draw[thick,bend right=-20,->] (w1) to (v2);
    \draw[thick,bend right=-30,->] (w1) to (v3);
    \draw[ultra thick,dashed,bend right=-20,->] (w2) to (v3);
    \draw[thick,bend right=-30,->] (w2) to (v4);
    \draw[thick,bend right=-20,->] (w3) to (v4);
    \draw[thick,bend right=-30,->] (w3) to (v5);
    \draw[thick,bend right=-20,->] (w4) to (v5);
    \draw[ultra thick,dashed,->] (w4) .. controls (0:5) and (270:4) .. (v0);
    \draw[thick,->] (w5) .. controls (300:3.5) and (240:3.5) .. (v0);
    \draw[thick,->] (w5) .. controls (270:3) .. (v1);    
  \end{tikzpicture}
  \caption{The network representation for $n=6$ and
    $(\alpha_t,\beta_t,\gamma_t,\delta_t)=(1,2,1,2)$ for all $t\in[0,n-1]$. The dashed cycle corresponds to $\vect
    y=(1,1,0,1,0,0)$ and $\vect z=(1,0,0,1,0,0)$.}
  \label{fig:network}
\end{figure}
In terms of switching sequences, an arc $((0,t),(1,l))$ corresponds to switching on in period $t$
and switching off in period $l$, and an arc $((1,t),(0,l))$ corresponds to switching off in period
$t$ and switching on in period $l$. Feasible switching sequences correspond to directed cycles of
length $n$ where the length of an arc $((i,t),\,(1-i,t+p))$ for $i\in\{0,1\}$ is $p$.
As in \cite{QueyranneWolsey2017} we can use the flow interpretation to obtain a formulation for
$\Z$ in the following way. For every node $v\in V$, let
$\Ain(v)$ and $\Aout(v)$ denote the sets of arcs entering and leaving $v$, respectively. For
convenience, we will omit one pair of brackets, whenever a node $(i,t)$ appears as an argument, that is, we
will write $\Ain(i,t)$ instead of $\Ain((i,t))$. For $t\in[0,n-1]$, we define
\begin{align*}
  \Aoff(t) &= \left\{\left((1,r),\,(0,l)\right)\in A\,:\,t\in[r,l-1]\}\right\},&
  \Aon(t) &= \left\{\left((0,r),\,(1,l)\right)\in A\,:\,t\in[r,l-1]\}\right\}.
\end{align*}
If $C$ is a cycle of length $n$, then for every $t\in[0,n-1]$, $C$ contains exactly one arc from
$\Aoff(t)\cup \Aon(t)$, and in the correspondence between cycles $C$ and vectors $(\vect y,\vect z)\in\Z$, we have
\begin{align*}
  y_t&=
       \begin{cases}
         0 &\text{if }C\text{ contains an arc }a\in \Aoff(t)\\
         1 &\text{if }C\text{ contains an arc }a\in \Aon(t)         
       \end{cases} &&t\in[0,n-1],\\
  z_t&=1\iff C\text{ contains an arc }a\in \Aout(0,t) &&t\in[0,n-1].
\end{align*}
Let $Q=Q(n,\vect\alpha,\vect\beta,\vect\gamma,\vect\delta)\subseteq\reals^{\lvert A\rvert+2n}$ be the
polytope defined by the constraints
\begin{align}
  \sum_{a\in \Aoff(0)\cup \Aon(0)}x_a &= 1,\label{eq:flow_injection}\\
  \sum_{a\in\Ain(v)}x_{a}-\sum_{a\in\Aout(v)}x_{a} &= 0 && v\in V, \label{eq:flow_conservation}\\
  y_t &= \sum_{a\in \Aon(t)}x_{a} && t\in[0,n-1],\label{eq:phi_1}\\
  z_t &= \sum_{a\in\Aout(0,t)}x_{a} && t\in[0,n-1],\label{eq:phi_2}\\
   x_a &\geq 0 && a\in A.\label{eq:nonnegativity}
\end{align}
\begin{proposition}\label{prop:Q_formulation}
  The polytope $Q$ is an extended formulation for $\Z$, that is, $\Z=\proj_{y,z}(Q)\cap\ints^{2n}$.
\end{proposition}
\begin{proof}
  For every $(\vect y,\vect z)\in\Z$ we have a corresponding cycle $C$ of length $n$. Let us define
  $\vect x\in\{0,1\}^{\lvert A\rvert}$ as $x_a=1\iff a\in C$. This provides a point $(\vect x,\vect y,\vect z)\in Q$, and
  shows $\Z\subseteq\proj_{y,z}(Q)\cap\ints^{2n}$. For the converse inclusion we start with an
  arbitrary $(\vect y,\,\vect z)\in\proj_{y,z}(Q)\cap\ints^{2n}$, and fix a vector
  $\vect x\in\reals^{\lvert A\rvert}$ with $(\vect x,\vect y,\vect z)\in Q$. We need to verify that
  $(\vect y,\vect z)$ satisfies~\eqref{eq:lower_bound_on} through~\eqref{eq:switch_on}. For this
  purpose the following observations are useful:
  \begin{align}
    \sum_{a\in \Aoff(t)\cup \Aon(t)}x_a &= 1 &&t\in[0,n-1],\label{eq:constant_flow}\\
        \sum_{a\in\Aout(0,t)}x_{a}&=
    \begin{cases}
      1 &\text{if }y_{t-1}=0\text{ and }y_t=1\\
      0 &\text{otherwise}.
    \end{cases}&&t\in[0,n-1],\label{eq:switch_on_Q}\\
    \sum_{a\in\Aout(1,t)}x_{a}&=
    \begin{cases}
      1 &\text{if }y_{t-1}=1\text{ and }y_t=0\\
      0 &\text{otherwise}.
    \end{cases}&&t\in[0,n-1].\label{eq:switch_off_Q}    
  \end{align}
  These observations can be seen as follows:
  \begin{description}
  \item[\eqref{eq:constant_flow}] Note that
  \begin{align*}
    \left(\Aoff(t)\cup \Aon(t)\right)\setminus\left(\Aoff(t-1)\cup \Aon(t-1)\right) &=
                                                                                \Aout(0,t)\cup\Aout(1,t),\\
    \left(\Aoff(t-1)\cup \Aon(t-1)\right)\setminus\left(\Aoff(t)\cup \Aon(t)\right) &= \Ain(0,t)\cup\Ain(1,t),
  \end{align*}
  and therefore,
  \begin{multline*}
    \sum_{a\in \Aoff(t)\cup \Aon(t)}x_a=\sum_{a\in \Aoff(t-1)\cup
      \Aon(t-1)}x_a+\sum_{a\in\Aout(0,t)\cup\Aout(1,t)}x_a-\sum_{a\in\Ain(0,t)\cup\Ain(1,t)}x_a\\
    \stackrel{\eqref{eq:flow_conservation}}{=}\sum_{a\in \Aoff(t-1)\cup \Aon(t-1)}x_a.
  \end{multline*}
Together with~\eqref{eq:flow_injection} and induction on $t$ this implies~\eqref{eq:constant_flow}.
\item[\eqref{eq:switch_on_Q}] With $\Aout(0,t)=\Aon(t)\setminus
  \Aon(t-1)$ we obtain
  \[\sum_{a\in\Aout(0,t)}x_{a}=\sum_{a\in \Aon(t)\setminus \Aon(t-1)}x_a\leq\sum_{a\in \Aon(t)}x_a\stackrel{\eqref{eq:phi_1}}{=}y_t,\]
  which implies $\sum_{a\in\Aout(0,t)}x_{a}=0$ if $y_t=0$. If $y_t=1$ and $y_{t-1}=0$, then
\[1=y_t-y_{t-1}\stackrel{\eqref{eq:phi_1}}{=}\sum_{a\in \Aon(t)}x_a-\sum_{a\in \Aon(t-1)}x_a\leq \sum_{a\in \Aon(t)\setminus \Aon(t-1)}x_a=\sum_{a\in\Aout(0,t)}x_{a}\]
and consequently, $\sum_{a\in\Aout(0,t)}x_{a}=1$. Finally, for $y_{t-1}=y_t=1$ we note
that $\Ain(0,t)\subseteq \Aoff(t-1)$ and therefore
\[\sum_{a\in\Aout(0,t)}x_a\stackrel{\eqref{eq:flow_conservation}}{=}\sum_{a\in\Ain(0,t)}x_a\leq\sum_{a\in
\Aoff(t-1)}x_a\stackrel{\eqref{eq:constant_flow}}{=}1-\sum_{a\in \Aon(t-1)}x_a\stackrel{\eqref{eq:phi_1}}{=}1-y_{t-1}=0.\]
\item[\eqref{eq:switch_off_Q}] With $\Aout(1,t)=\Aoff(t)\setminus
  \Aoff(t-1)$ we obtain  
  \[\sum_{a\in\Aout(1,t)}x_{a}=\sum_{a\in \Aoff(t)\setminus \Aoff(t-1)}x_a\leq\sum_{a\in \Aoff(t)}x_a\stackrel{\eqref{eq:constant_flow}}{=}1-\sum_{a\in \Aon(t)}x_a\stackrel{\eqref{eq:phi_1}}{=}1-y_t,\]
  which implies $\sum_{a\in\Aout(1,t)}x_{a}=0$ if $y_t=1$. If $y_t=0$ and $y_{t-1}=1$, then
  \begin{multline*}
  1=y_{t-1}-y_{t}\stackrel{\eqref{eq:phi_1}}{=}\sum_{a\in \Aon(t-1)}x_a-\sum_{a\in
    \Aon(t)}x_a\stackrel{\eqref{eq:constant_flow}}{=}\left(1-\sum_{a\in
      \Aoff(t-1)}x_a\right)-\left(1-\sum_{a\in \Aoff(t)}x_a\right)\\
  =\sum_{a\in \Aoff(t)}x_a-\sum_{a\in
      \Aoff(t-1)}x_a\leq \sum_{a\in \Aoff(t)\setminus \Aoff(t-1)}x_a=\sum_{a\in\Aout(1,t)}x_{a}  
  \end{multline*}
and consequently, $\sum_{a\in\Aout(1,t)}x_{a}=1$. Finally, for $y_{t-1}=y_t=0$ we note
that $\Ain(1,t)\subseteq \Aon(t-1)$ and therefore
\[\sum_{a\in\Aout(1,t)}x_a\stackrel{\eqref{eq:flow_conservation}}{=}\sum_{a\in\Ain(1,t)}x_a\leq\sum_{a\in
\Aon(t-1)}x_a\stackrel{\eqref{eq:phi_1}}{=}y_{t-1}=0.\]
\end{description}
After establishing~\eqref{eq:constant_flow},~\eqref{eq:switch_on_Q} and~\eqref{eq:switch_off_Q}, we can
now proceed to verify~\eqref{eq:lower_bound_on} through~\eqref{eq:switch_on}.
\begin{description}
\item[\eqref{eq:lower_bound_on}] Suppose $y_t-y_{t-1}=1$, that is, $y_t=1$ and $y_{t-1}=0$, and fix
  $i\in[0,\alpha_t-1]$. Using $\Aout(0,t)\subseteq
  \Aon(t+i)$, we obtain
\[y_{t+i}\stackrel{\eqref{eq:phi_1}}{=}\sum_{a\in
    \Aon(t+i)}x_a\geq\sum_{a\in\Aout(0,t)}x_a\stackrel{\eqref{eq:switch_on_Q}}{=}1.\]
Now \eqref{eq:constant_flow} implies $y_{t+i}\leq 1$, and we conclude $y_{t+i}=1$, as required.
\item[\eqref{eq:upper_bound_on}] Suppose $y_t-y_{t-1}=1$, that is, $y_t=1$ and
  $y_{t-1}=0$. Then~\eqref{eq:switch_on_Q} implies $\sum_{a\in\Aout(0,t)}x_a=1$. In particular,
  $x_a>0$ for some
  $a\in\Aout(0,t)=\left\{\left((0,t),(1,t+i)\right)\,:\,i\in[\alpha_t,\beta_t]\right\}$. Fix an
  $i\in[\alpha_t,\beta_t]$ such that $x_{a^*}>0$ for $a^*=((0,t),(1,t+i))$, and note that $a^*\in
  \Aoff(t+i)$. Then
  \[y_{t+i}\stackrel{\eqref{eq:phi_1}}{=}\sum_{a\in
      \Aon(t+i)}x_a\stackrel{\eqref{eq:constant_flow}}{=}1-\sum_{a\in \Aoff(t+i)}x_a\leq
    1-x_{a^*}<1,\]
  and by integrality we conclude $y_{t+i}=0$, as required.
\item[\eqref{eq:lower_bound_off}]
  Suppose $y_{t-1}-y_{t}=1$, that is, $y_t=0$ and $y_{t-1}=1$, and fix
  $i\in[0,\gamma_t-1]$. Using $\Aout(1,t)\subseteq
  \Aoff(t+i)$, we obtain
\[y_{t+i}\stackrel{\eqref{eq:phi_1}}{=}\sum_{a\in
    \Aon(t+i)}x_a\stackrel{~\eqref{eq:constant_flow}}{=}1-\sum_{a\in
    \Aoff(t+i)}x_a\leq 1-\sum_{a\in\Aout(1,t)}x_a\stackrel{\eqref{eq:switch_off_Q}}{=}0.\]
Now $y_{t+i}\geq 0$ is a consequence of~\eqref{eq:nonnegativity} and~\eqref{eq:phi_1}, and we conclude $y_{t+i}=0$, as required.
\item[\eqref{eq:upper_bound_off}] Suppose $y_{t-1}-y_{t}=1$, that is, $y_t=0$ and
  $y_{t-1}=1$. Then~\eqref{eq:switch_on_Q} implies $\sum_{a\in\Aout(1,t)}x_a=1$. In particular,
  $x_a>0$ for some
  $a\in\Aout(1,t)=\left\{\left((1,t),(0,t+i)\right)\,:\,i\in[\gamma_t,\delta_t]\right\}$. Fix an
  $i\in[\gamma_t,\delta_t]$ such that $x_{a^*}>0$ for $a^*=((1,t),(0,t+i))$, and note that $a^*\in
  \Aon(t+i)$. Then
  \[y_{t+i}\stackrel{\eqref{eq:phi_1}}{=}\sum_{a\in
      \Aon(t+i)}x_a\geq x_{a^*}>0,\]
  and by integrality we conclude $y_{t+i}=1$, as required.
\item[\eqref{eq:switch_on}] This follows immediately from~\eqref{eq:switch_on_Q} and~\eqref{eq:phi_2}.\qedhere
\end{description}
\end{proof}
In the non-cyclic case, the polytope corresponding to $Q$ is integral and and its projection onto
the $(y,z)$ space gives $\conv(\Z)$ (\cite[Theorem
1]{QueyranneWolsey2017}). Unfortunately, this breaks down in the cyclic case, as the following
example shows.
\begin{example}
  Let $n=6$, $(\alpha_t,\,\beta_t,\,\gamma_t,\,\delta_t)=(1,3,1,3)$, for all $t\in[0,5]$. Then $Q$ is not integral, as
  the cycle $(0,0),(1,3),(0,4),(1,1),(0,2),(1,5),(0,0)$, shown in Figure, with coefficient $1/2$
  corresponds to an extreme point of $Q$. This point projects to $\vect y=(1,1/2,1,1/2,1,1/2)$,
  $\vect z=(1/2,0,1/2,0,1/2,0)$, which is not contained in $\conv(\Z)$.
\begin{figure}[htb]
  \centering
  \begin{tikzpicture}[scale=1.2]
    \draw[use as bounding box,color=white](-4.8,-1)rectangle(4.8,3.6);
    \node[ellipse,inner sep=.3pt,outer sep=.2pt,draw] (v0) at (180:2) {{\small $(0,0)$}};
    \node[ellipse,inner sep=.3pt,outer sep=.2pt,draw] (v1) at (144:1.7) {{\small $(0,1)$}};
    \node[ellipse,inner sep=.3pt,outer sep=.2pt,draw] (v2) at (108:1.6) {{\small $(0,2)$}};
    \node[ellipse,inner sep=.3pt,outer sep=.2pt,draw] (v3) at (72:1.6) {{\small $(0,3)$}};
    \node[ellipse,inner sep=.3pt,outer sep=.2pt,draw] (v4) at (36:1.7) {{\small $(0,4)$}};
    \node[ellipse,inner sep=.3pt,outer sep=.2pt,draw] (v5) at (0:2) {{\small $(0,5)$}};
    \node[ellipse,inner sep=.3pt,outer sep=.2pt,draw] (w0) at (180:4) {{\small $(1,0)$}};
    \node[ellipse,inner sep=.3pt,outer sep=.2pt,draw] (w1) at (144:3.7) {{\small $(1,1)$}};
    \node[ellipse,inner sep=.3pt,outer sep=.2pt,draw] (w2) at (108:3.5) {{\small $(1,2)$}};
    \node[ellipse,inner sep=.3pt,outer sep=.2pt,draw] (w3) at (72:3.5) {{\small $(1,3)$}};
    \node[ellipse,inner sep=.3pt,outer sep=.2pt,draw] (w4) at (36:3.7) {{\small $(1,4)$}};
    \node[ellipse,inner sep=.3pt,outer sep=.2pt,draw] (w5) at (0:4) {{\small $(1,5)$}};
    \draw[thick,bend right=-45,->] (v0) to (w3);
    \draw[thick,bend right=-20,->] (w3) to (v4);
    \draw[thick,->] (v4) .. controls (240:2) and (180:5) .. (w1);
    \draw[thick,bend right=-20,->] (w1) to (v2);
    \draw[thick,bend right=-45,->] (v2) to (w5);
    \draw[thick,bend left=30,->] (w5) to (v0); 
  \end{tikzpicture}
  \caption{A cycle corresponding to a fractional extreme point for
    $(\alpha_t,\beta_t,\gamma_t,\delta_t)=(1,3,1,3)$ for all $t\in[0,5]$.}
  \label{fig:example}
\end{figure}
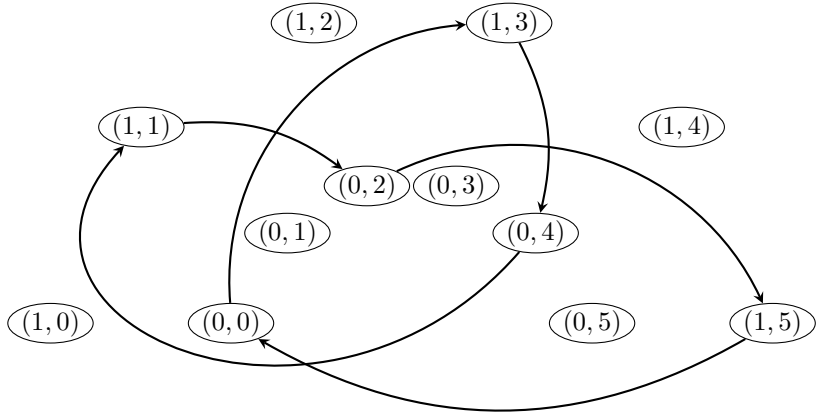
\end{example}
It is still possible to obtain an extended formulation for $\conv(\Z)$ as a flow problem in a
network of size polynomial in $n$. For this purpose we make copies of the original network: one for
every node $(i,\tau)$ such that at least one arc in $\Aout(i,\tau)$ ``wraps around''. In other
words, there is a copy for node $(0,\tau)$ if $\tau+\beta_\tau\geq n$, and there is a copy for node
$(1,\tau)$ if $\tau+\delta_\tau\geq n$. We also add an origin node $O$ and a destination node
$D$. The underlying idea is that $O$-$D$-paths through the copy of the network for node $(0,\tau)$
when $\tau+\beta_\tau\geq n$ represent cycles using an arc of the form $((0,\tau),(1,\tau+p-n))$
with $\max\{\alpha_\tau,n-\tau\}\leq p\leq\beta_\tau$, and $O$-$D$-paths through the copy for node
$(1,\tau)$ when $\tau+\delta_\tau\geq n$ represent cycles using an arc of the form $((1,\tau),(0,\tau+q-n))$ with
$\max\{\gamma_\tau,n-\tau\}\leq q\leq\delta_\tau$. More formally, with
$T_0=\{\tau\in[0,n-1]\,:\,\tau+\beta_\tau\geq n\}$ and
$T_1=\{\tau\in[0,n-1]\,:\,\tau+\delta_\tau\geq n\}$ the extended network has node set
\[ V'=\{O,D\}\cup\left\{(i,t,j,\tau)\,:\,i\in\{0,1\},\,t\in[0,n-1],\,j\in\{0,1\},\,\tau\in T_j\right\},\]
and arc set $A'=A'_1\cup\dots\cup A'_5$, where
\begin{align*}
  A'_1 &= \left\{\left(O,(1,t,0,\tau)\right)\ :\ \tau\in T_0,\ t=\tau+p-n\text{ for some
         }p\in[\alpha_{\tau},\beta_{\tau}]\right\},\\
  A'_2 &= \left\{\left(O,(0,t,1,\tau)\right)\ :\ \tau\in T_1,\ t=\tau+q-n\text{ for some
         }q\in[\gamma_{\tau},\delta_{\tau}]\right\},\\
  A'_3 &= \left\{\left((0,\tau,0,\tau),D\right)\ :\ \tau\in T_0\right\}\cup \left\{\left((1,\tau,1,\tau),D\right)\ :\ \tau\in T_1\right\}, \\
  A'_4 &= \left\{\left((i,t,0,\tau),(1-i,l,0,\tau)\right)\ :\
         \left((i,t),(1-i,l)\right)\in A,\,t<l<n,\ \tau\in T_0\right\},\\
  A'_5 &= \left\{\left((i,t,1,\tau),(1-i,l,1,\tau)\right)\ :\
         \left((i,t),(1-i,l)\right)\in A,\,t<l<n,\ \tau\in T_1\right\}.
\end{align*}
The network for $n=6$, $(\alpha_t,\beta_t,\gamma_t,\delta_t)=(1,2,1,2)$ for all $t\in[0,n-1]$, which implies $T_0=T_1=\{4,5\}$, is shown in Figure~\ref{fig:expanded_network}.
\begin{figure}[htb]
  \centering
  \begin{tikzpicture}
    \node[ellipse,inner sep=1pt,outer sep=.2pt,draw] (O) at (0,0) {{\small $O$}};
    \node[ellipse,inner sep=1pt,outer sep=.2pt,draw] (D) at (14,0) {{\small $D$}};
    \node[ellipse,inner sep=1.2pt,outer sep=.2pt,color=lightgray,draw] (v0u5) at (2,.5) {{\small $0005$}};
    \node[ellipse,inner sep=1.2pt,outer sep=.2pt,draw] (w0u5) at (2,1.5) {{\small $1005$}};
    \node[ellipse,inner sep=1.2pt,outer sep=.2pt,color=lightgray,draw] (v0u4) at (2,2.5) {{\small $0004$}};
    \node[ellipse,inner sep=1.2pt,outer sep=.2pt,draw] (w0u4) at (2,3.5) {{\small $1004$}};
    \node[ellipse,inner sep=1.2pt,outer sep=.2pt,draw] (v0d5) at (2,-1.5) {{\small $0015$}};
    \node[ellipse,inner sep=1.2pt,outer sep=.2pt,color=lightgray,draw] (w0d5) at (2,-.5) {{\small $1015$}};
    \node[ellipse,inner sep=1.2pt,outer sep=.2pt,draw] (v0d4) at (2,-3.5) {{\small $0014$}};
    \node[ellipse,inner sep=1.2pt,outer sep=.2pt,color=lightgray,draw] (w0d4) at (2,-2.5) {{\small $1014$}};
    \node[ellipse,inner sep=1.2pt,outer sep=.2pt,draw] (v1u5) at (4,.5) {{\small $0105$}};
    \node[ellipse,inner sep=1.2pt,outer sep=.2pt,draw] (w1u5) at (4,1.5) {{\small $1105$}};
    \node[ellipse,inner sep=1.2pt,outer sep=.2pt,draw] (v1u4) at (4,2.5) {{\small $0104$}};
    \node[ellipse,inner sep=1.2pt,outer sep=.2pt,color=lightgray,draw] (w1u4) at (4,3.5) {{\small $1104$}};
    \node[ellipse,inner sep=1.2pt,outer sep=.2pt,draw] (v1d5) at (4,-1.5) {{\small $0115$}};
    \node[ellipse,inner sep=1.2pt,outer sep=.2pt,draw] (w1d5) at (4,-.5) {{\small $1115$}};
    \node[ellipse,inner sep=1.2pt,outer sep=.2pt,color=lightgray,draw] (v1d4) at (4,-3.5) {{\small $0114$}};
    \node[ellipse,inner sep=1.2pt,outer sep=.2pt,draw] (w1d4) at (4,-2.5) {{\small $1114$}};
    \node[ellipse,inner sep=1.2pt,outer sep=.2pt,draw] (v2u5) at (6,.5) {{\small $0205$}};
    \node[ellipse,inner sep=1.2pt,outer sep=.2pt,draw] (w2u5) at (6,1.5) {{\small $1205$}};
    \node[ellipse,inner sep=1.2pt,outer sep=.2pt,draw] (v2u4) at (6,2.5) {{\small $0204$}};
    \node[ellipse,inner sep=1.2pt,outer sep=.2pt,draw] (w2u4) at (6,3.5) {{\small $1204$}};
    \node[ellipse,inner sep=1.2pt,outer sep=.2pt,draw] (v2d5) at (6,-1.5) {{\small $0215$}};
    \node[ellipse,inner sep=1.2pt,outer sep=.2pt,draw] (w2d5) at (6,-.5) {{\small $1215$}};
    \node[ellipse,inner sep=1.2pt,outer sep=.2pt,draw] (v2d4) at (6,-3.5) {{\small $0214$}};
    \node[ellipse,inner sep=1.2pt,outer sep=.2pt,draw] (w2d4) at (6,-2.5) {{\small $1214$}};
    \node[ellipse,inner sep=1.2pt,outer sep=.2pt,draw] (v3u5) at (8,.5) {{\small $0305$}};
    \node[ellipse,inner sep=1.2pt,outer sep=.2pt,draw] (w3u5) at (8,1.5) {{\small $1305$}};
    \node[ellipse,inner sep=1.2pt,outer sep=.2pt,color=lightgray,draw] (v3u4) at (8,2.5) {{\small $0304$}};
    \node[ellipse,inner sep=1.2pt,outer sep=.2pt,draw] (w3u4) at (8,3.5) {{\small $1304$}};
    \node[ellipse,inner sep=1.2pt,outer sep=.2pt,draw] (v3d5) at (8,-1.5) {{\small $0315$}};
    \node[ellipse,inner sep=1.2pt,outer sep=.2pt,draw] (w3d5) at (8,-.5) {{\small $1315$}};
    \node[ellipse,inner sep=1.2pt,outer sep=.2pt,draw] (v3d4) at (8,-3.5) {{\small $0314$}};
    \node[ellipse,inner sep=1.2pt,outer sep=.2pt,color=lightgray,draw] (w3d4) at (8,-2.5) {{\small $1314$}};
    \node[ellipse,inner sep=1.2pt,outer sep=.2pt,color=lightgray,draw] (v4u5) at (10,.5) {{\small $0405$}};
    \node[ellipse,inner sep=1.2pt,outer sep=.2pt,draw] (w4u5) at (10,1.5) {{\small $1405$}};
    \node[ellipse,inner sep=1.2pt,outer sep=.2pt,draw] (v4u4) at (10,2.5) {{\small $0404$}};
    \node[ellipse,inner sep=1.2pt,outer sep=.2pt,color=lightgray,draw] (w4u4) at (10,3.5) {{\small $1404$}};
    \node[ellipse,inner sep=1.2pt,outer sep=.2pt,draw] (v4d5) at (10,-1.5) {{\small $0415$}};
    \node[ellipse,inner sep=1.2pt,outer sep=.2pt,color=lightgray,draw] (w4d5) at (10,-.5) {{\small $1415$}};
    \node[ellipse,inner sep=1.2pt,outer sep=.2pt,color=lightgray,draw] (v4d4) at (10,-3.5) {{\small $0414$}};
    \node[ellipse,inner sep=1.2pt,outer sep=.2pt,draw] (w4d4) at (10,-2.5) {{\small $1414$}};
    \node[ellipse,inner sep=1.2pt,outer sep=.2pt,draw] (v5u5) at (12,.5) {{\small $0505$}};
    \node[ellipse,inner sep=1.2pt,outer sep=.2pt,color=lightgray,draw] (w5u5) at (12,1.5) {{\small $1505$}};
    \node[ellipse,inner sep=1.2pt,outer sep=.2pt,color=lightgray,draw] (v5u4) at (12,2.5) {{\small $0504$}};
    \node[ellipse,inner sep=1.2pt,outer sep=.2pt,color=lightgray,draw] (w5u4) at (12,3.5) {{\small $1504$}};
    \node[ellipse,inner sep=1.2pt,outer sep=.2pt,color=lightgray,draw] (v5d5) at (12,-1.5) {{\small $0515$}};
    \node[ellipse,inner sep=1.2pt,outer sep=.2pt,draw] (w5d5) at (12,-.5) {{\small $1515$}};
    \node[ellipse,inner sep=1.2pt,outer sep=.2pt,color=lightgray,draw] (v5d4) at (12,-3.5) {{\small $0514$}};
    \node[ellipse,inner sep=1.2pt,outer sep=.2pt,color=lightgray,draw] (w5d4) at (12,-2.5) {{\small $1514$}};
    \draw[ultra thick,dashed,bend right=15,->] (O) to (v0d4);
    \draw[thick,->] (O) to (v0d5);
    \draw[thick,bend right=10,->] (O) to (v1d5);
    \draw[thick,bend left=15,->] (O) to (w0u4);
    \draw[thick,->] (O) to (w0u5);
    \draw[thick,bend left=10,->] (O) to (w1u5);
    \draw[ultra thick,dashed,bend right=30,->] (w4d4) to (D);
    \draw[thick,->] (w5d5) to (D);
    \draw[thick,bend left=30,->] (v4u4) to (D);
    \draw[thick,->] (v5u5) to (D);
    \draw[thick,->] (v0d4) to (w1d4);
    \draw[ultra thick,dashed,->] (v0d4) to (w2d4);
    \draw[thick,color=lightgray,->] (v1d4) to (w2d4);
    \draw[thick,color=lightgray,->] (v1d4) to (w3d4);
    \draw[thick,color=lightgray,->] (v2d4) to (w3d4);
    \draw[thick,->] (v2d4) to (w4d4);
    \draw[ultra thick,dashed,->] (v3d4) to (w4d4);
    \draw[thick,color=lightgray,->] (v3d4) to (w5d4);
    \draw[thick,color=lightgray,->] (v4d4) to (w5d4);
    \draw[thick,color=lightgray,->] (w0d4) to (v1d4);
    \draw[thick,color=lightgray,->] (w0d4) to (v2d4);
    \draw[thick,->] (w1d4) to (v2d4);
    \draw[thick,->] (w1d4) to (v3d4);
    \draw[ultra thick,dashed,->] (w2d4) to (v3d4);
    \draw[thick,color=lightgray,->] (w2d4) to (v4d4);
    \draw[thick,color=lightgray,->] (w3d4) to (v4d4);
    \draw[thick,color=lightgray,->] (w3d4) to (v5d4);
    \draw[thick,color=lightgray,->] (w4d4) to (v5d4);
    \draw[thick,->] (v0d5) to (w1d5);
    \draw[thick,->] (v0d5) to (w2d5);
    \draw[thick,->] (v1d5) to (w2d5);
    \draw[thick,->] (v1d5) to (w3d5);
    \draw[thick,->] (v2d5) to (w3d5);
    \draw[thick,color=lightgray,->] (v2d5) to (w4d5);
    \draw[thick,color=lightgray,->] (v3d5) to (w4d5);
    \draw[thick,->] (v3d5) to (w5d5);
    \draw[thick,->] (v4d5) to (w5d5);
    \draw[thick,color=lightgray,->] (w0d5) to (v1d5);
    \draw[thick,color=lightgray,->] (w0d5) to (v2d5);
    \draw[thick,->] (w1d5) to (v2d5);
    \draw[thick,->] (w1d5) to (v3d5);
    \draw[thick,->] (w2d5) to (v3d5);
    \draw[thick,->] (w2d5) to (v4d5);
    \draw[thick,->] (w3d5) to (v4d5);
    \draw[thick,color=lightgray,->] (w3d5) to (v5d5);
    \draw[thick,color=lightgray,->] (w4d5) to (v5d5);
    \draw[thick,color=lightgray,->] (v0u4) to (w1u4);
    \draw[thick,color=lightgray,->] (v0u4) to (w2u4);
    \draw[thick,->] (v1u4) to (w2u4);
    \draw[thick,->] (v1u4) to (w3u4);
    \draw[thick,->] (v2u4) to (w3u4);
    \draw[thick,color=lightgray,->] (v2u4) to (w4u4);
    \draw[thick,color=lightgray,->] (v3u4) to (w4u4);
    \draw[thick,color=lightgray,->] (v3u4) to (w5u4);
    \draw[thick,color=lightgray,->] (v4u4) to (w5u4);
    \draw[thick,->] (w0u4) to (v1u4);
    \draw[thick,->] (w0u4) to (v2u4);
    \draw[thick,color=lightgray,->] (w1u4) to (v2u4);
    \draw[thick,color=lightgray,->] (w1u4) to (v3u4);
    \draw[thick,color=lightgray,->] (w2u4) to (v3u4);
    \draw[thick,->] (w2u4) to (v4u4);
    \draw[thick,->] (w3u4) to (v4u4);
    \draw[thick,color=lightgray,->] (w3u4) to (v5u4);
    \draw[thick,color=lightgray,->] (w4u4) to (v5u4);
    \draw[thick,color=lightgray,->] (v0u5) to (w1u5);
    \draw[thick,color=lightgray,->] (v0u5) to (w2u5);
    \draw[thick,->] (v1u5) to (w2u5);
    \draw[thick,->] (v1u5) to (w3u5);
    \draw[thick,->] (v2u5) to (w3u5);
    \draw[thick,->] (v2u5) to (w4u5);
    \draw[thick,->] (v3u5) to (w4u5);
    \draw[thick,color=lightgray,->] (v3u5) to (w5u5);
    \draw[thick,color=lightgray,->] (v4u5) to (w5u5);
    \draw[thick,->] (w0u5) to (v1u5);
    \draw[thick,->] (w0u5) to (v2u5);
    \draw[thick,->] (w1u5) to (v2u5);
    \draw[thick,->] (w1u5) to (v3u5);
    \draw[thick,->] (w2u5) to (v3u5);
    \draw[thick,color=lightgray,->] (w2u5) to (v4u5);
    \draw[thick,color=lightgray,->] (w3u5) to (v4u5);
    \draw[thick,->] (w3u5) to (v5u5);
    \draw[thick,->] (w4u5) to (v5u5); 
  \end{tikzpicture}
  \caption{The expanded network for $n=6$ and $(\alpha_t,\beta_t,\gamma_t,\delta_t)=(1,2,1,2)$ for
    all $t\in[0,5]$, where we have omitted brackets and commas in the node labels. The light parts
    do not lie on any $O$-$D$-path and so can be eliminated from the network in a preprocessing
    step. The dashed path corresponds to the cycle in Figure~\ref{fig:network}.}
  \label{fig:expanded_network}
\end{figure}
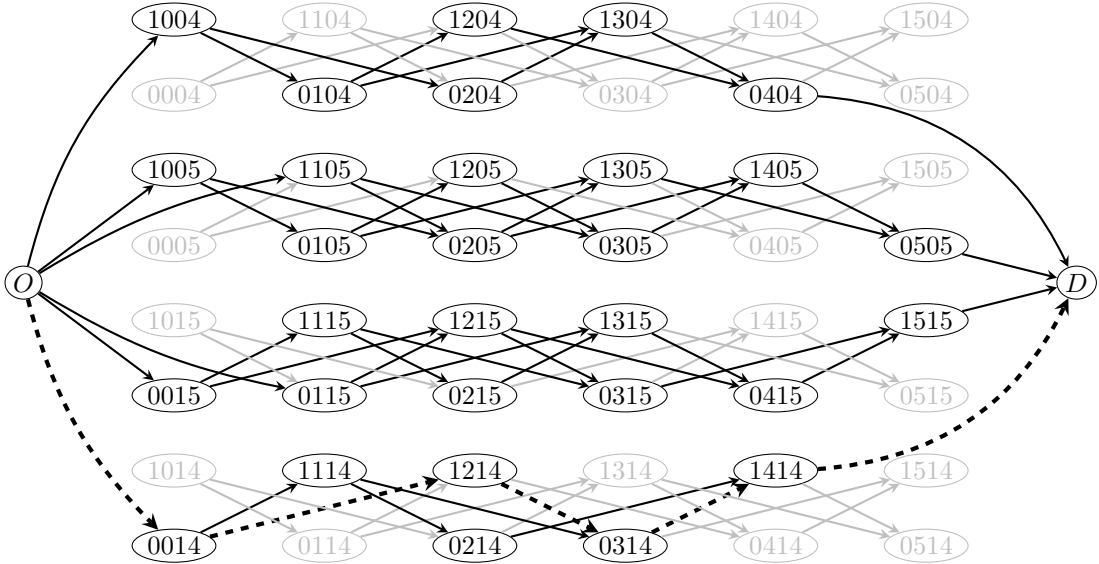
We define $\Aon'(t)$ to be the set of arcs corresponding to $y_t=1$:
\begin{multline*}
\Aon'(t)=\left\{\left(O,(1,l,0,\tau)\right)\in A'_1\,:\,l>t\right\}\cup\left\{\left((0,\tau,0,\tau),D\right)\in
  A'_3\,:\,\tau\leq
  t\right\}\\ \cup\left\{\left((0,k,i,\tau),(1,l,i,\tau)\right)\in
  A'_4\cup A'_5\,:\,k\leq t<l\right\},
\end{multline*}
and then we define the polytope $Q'=Q'(n,\vect\alpha,\vect\beta,\vect\gamma,\vect\delta)\subseteq\reals^{\lvert A'\rvert+2n}$ by the following constraints:
\begin{align}
  \sum_{a\in\Aout(O)}x'_a &= 1,\label{eq:flow_injection_extended}\\
  \sum_{a\in\Ain(v)}x'_{a}-\sum_{a\in\Aout(v)}x'_{a} &= 0 && v\in V'\setminus\{O,D\}, \label{eq:flow_conservation_extended}\\
  y_t &= \sum_{a\in\Aon'(t)}x'_{a} && t\in[0,n-1],\label{eq:y-trafo}\\
  z_t &= \sum_{\tau\in T_1}\sum_{a\in\Aout\left(0,t,0,\tau\right)}x'_{a}+\sum_{\tau\in T_2}\sum_{a\in\Aout\left(0,t,1,\tau\right)}x'_{a} && t\in[0,n-1],\label{eq:z-trafo}\\
   x'_a &\geq 0 && a\in A'.\label{eq:nonnegativity_extended}
\end{align}
\begin{proposition}
The polytope $Q'$ is integral and $\conv(\Z)=\proj_{y,z}(Q')$.
\end{proposition}
\begin{proof}
  The polytope $\proj_x(Q')$, described by~\eqref{eq:flow_injection_extended},
  \eqref{eq:flow_conservation_extended} and~\eqref{eq:nonnegativity_extended} is integral as the
  constraint matrix is a network matrix, hence totally unimodular. Constraints~\eqref{eq:y-trafo}
  and~\eqref{eq:z-trafo} preserve integrality, because they only write the $y$- and $z$- variables as
  linear combinations of the $x$-variables with integer coefficients. In order to see that the
  projection of $Q'$ is the convex hull of $\Z$ it is sufficient to note the one-to-one
  correspondence between elements of $\Z$ and $O$-$D$-paths in the network $(V',A')$.
\end{proof}
\begin{corollary}
  The polytope $Q'$ provides a compact extended formulation for $\conv(\Z)$ with $O(n^3)$ variables
  and $O(n^2)$ constraints. Moreover, if the parameters $\beta_t$ and $\delta_t$ are $O(1)$ for all $t\in[0,n-1]$, then this
  reduces to $O(n)$ variables and constraints.
\end{corollary}
\begin{proof}
  The original network $(V,A)$ has $O(n)$ nodes, and since $(V',A')$ is constructed from $O(n)$
  copies of $(V,A)$, it has $O(n^2)$ nodes and all nodes, except possibly $O$ have degree
  $O(n)$. This implies $\lvert A'\rvert=O(n^3)$. If the parameters $\beta_t$ and $\delta_t$ are
  bounded for all $t\in[0,n-1]$ then there are only $O(1)$ copies, and every node has degree $O(1)$.
\end{proof}

\section{Towards a tight formulation in the $(y,z)$-space}\label{sec:yz-formulation}
Following~\cite{QueyranneWolsey2017}, we now assume that every
$\vect\eps\in\{\vect\alpha,\vect\beta,\vect\gamma,\vect\delta\}$ satisfies the \emph{weak
  monotonicity} condition: for every $t\in[0,n-1]$, $\eps_{t+1}\geq\eps_t-1$. This implies that by
waiting one period, one cannot be forced to switch on or off earlier. In particular, weak
monotonicity guarantees the existence of numbers $s(\vect\eps,t)\in[0,n-1]$ for every
$\vect\eps\in\{\vect\alpha,\vect\beta,\vect\gamma,\vect\delta\}$ and $t\in[0,n-1]$ such that
\[\left\{k\in[0,n-1]\,:\,t\in[k,k+\eps_{k}-1]\right\}=\left[s(\vect\eps,t),t\right].\]
For instance, (a) the interval $[s(\vect\alpha,t),t]$ is the set of time periods $k$ for which
$z_k=1$ implies $y_t=1$, and (b) $y_t=1$ implies that $z_k=1$ for some $k\in[s(\vect\beta,t),t]$.
\subsection{A formulation}\label{subsec:formulation}
Following the approach taken in~\cite[Section 3.1]{QueyranneWolsey2017} we define a polytope
$P=P(n,\vect\alpha,\vect\beta,\vect\gamma,\vect\delta)\subseteq\reals^{2n}$ by
\begin{align}
  z_t &\geq y_t-y_{t-1} &&t\in[0,n-1],\label{eq:switch-on}\\
  \sum_{k\in\left[s(\vect\alpha,t),t\right]}z_{k}&\leq y_t
                        &&t\in[0,n-1], \label{eq:on-lower_bound}\\
  y_t&\leq \sum_{k\in\left[s(\vect\beta,t),t\right]}z_{k}
                        &&t\in[0,n-1], \label{eq:on-upper_bound}\\
  \sum_{k\in[s(\vect\gamma,t),t]}z_{k}&\leq 1-y_{s(\vect\gamma,t)-1}
                        &&t\in[0,n-1], \label{eq:off-lower_bound}\\
  1-y_{s(\vect\delta,t)-1}&\leq\sum_{k\in[s(\vect\delta,t),t]}z_{k}
                        &&t\in[0,n-1], \label{eq:off-upper_bound}\\
  0\leq y_t,z_t &\leq 1 &&t\in[0,n-1].\label{eq:variable-bounds}
\end{align}
Analogous to Proposition 2 in~\cite{QueyranneWolsey2017}, we find that this provides a
formulation for $\conv(\Z)$.
\begin{proposition}\label{prop:formulation}
  The polytope $P$ is a formulation for $\Z$, that is, $\Z=P\cap\ints^{2n}$.
\end{proposition}
\begin{proof}
  First, we start with an arbitrary $(\vect y,\vect z)\in\Z$, and verify that it
  satisfies~\eqref{eq:switch-on} through~\eqref{eq:variable-bounds}.
  \begin{description}
  \item[\eqref{eq:switch-on}] This follows immediately from~\eqref{eq:switch_on}.
  \item[\eqref{eq:on-lower_bound}] Suppose the inequality is violated for
    some $t\in[0,n-1]$. Then $y_t=0$ and $z_{k}=1$ for some
    $k\in[s(\vect\alpha,t),t]$. Using~\eqref{eq:switch_on}, this implies $y_{k}=1$ and
    $y_{k-1}=0$, and then by~\eqref{eq:lower_bound_on}, $y_{l}=1$ for all
    $l\in[k,k+\alpha_{k}-1]$, and in particular, $y_t=1$, which is the required
    contradiction. 
  \item[\eqref{eq:on-upper_bound}] Let $t\in[0,n-1]$ be an index with $y_t=1$ and let
    $k\in[0,n-1]$ be the unique index with $y_{k-1}=0$ and $y_{l}=1$ for all $l\in[k,t]$. Then
    $z_{k}=1$ by~\eqref{eq:switch_on}, and using~\eqref{eq:upper_bound_on},
    $t\in[k,k+\beta_{k}-1]$, hence $k\in[s(\vect\beta,t),t]$.
  \item[\eqref{eq:off-lower_bound}] Suppose the inequality is violated for some $t\in[0,n-1]$. Then
    $y_{s(\gamma,t)-1}=1$ and $z_{k}=1$ for some $k\in[s(\vect\gamma,t),t]$. Using~\eqref{eq:switch_on}, this implies $y_{k}=1$ and
    $y_{k-1}=0$. Let $l$ be the first index in the sequence
    $(s(\vect\gamma,t),s(\vect\gamma,t)+1,\dots,k-1)$ with $y_l=0$. By~(\ref{eq:lower_bound_off}),
    $y_{l+i}=0$ for all $i\in[0,\gamma_l-1]$. By weak monotonicity, $k\in[l,l+\gamma_l-1]$, and this
    is the required contradiction.
  \item[\eqref{eq:off-upper_bound}] Let $t\in[0,n-1]$ be a period with $y_{s(\delta,t)-1}=0$, and set
    $l=s(\vect\delta,t)-1$. By~(\ref{eq:upper_bound_off}), $y_{l+i}=1$ for some
    $i\in[1,\delta_l]$. Now~(\ref{eq:switch_on}) implies that $z_{l+i}=1$ for some
    $i\in[1,\delta_l]$. By definition of $s(\vect\delta,t)$,
    $\{l+i\,:\,i\in[1,\delta_l]\}\subseteq\{l+1,l+2,\dots,t\}$, hence $z_{l+1}+z_{l+2}+\dots+z_t\geq
    1$, as required.
  \item[\eqref{eq:variable-bounds}] By definition. 
  \end{description}
  We have shown that $(\vect y,\vect z)\in P$, and therefore $\Z\subseteq P\cap\ints^{2n}$. For the
  reverse inclusion, we start with an arbitrary $(\vect y,\vect z)\in P\cap\ints^{2n}$ and verify
  that it satisfies~\eqref{eq:lower_bound_on} through~\eqref{eq:switch_on}.
  \begin{description}
  \item[\eqref{eq:switch_on}] It follows immediately from~\eqref{eq:switch-on} that if
    $y_t=1$ and $y_{t-1}=0$ then $z_t=1$. For the converse, let $t\in[0,n-1]$ be a period
    with $z_t=1$, and let $k=s(\vect\gamma,t)$. It follows from~(\ref{eq:on-lower_bound}) that
    $y_{t}=1$. From~(\ref{eq:off-lower_bound}), we obtain $y_{k-1}=0$. Let $l$ be the first index in
    $[k,t]$ with $y_l=1$. Then~(\ref{eq:switch-on}) implies $z_l=1$, and if
    $l\neq t$, the left hand side of~(\ref{eq:off-lower_bound}) is at least $2$. We conclude $l=t$,
    hence $y_{t-1}=0$.
  \item[\eqref{eq:lower_bound_on}] We use~\eqref{eq:switch-on} and~\eqref{eq:on-lower_bound} (note
    that $t\in[s(\vect\alpha,t+i),t+i]$ for all $i\in[0,\alpha_t-1]$):
    \[y_t-y_{t-1}=1\stackrel{\eqref{eq:switch-on}}{\implies}
      z_t=1\stackrel{\eqref{eq:on-lower_bound}}{\implies} y_{t+i}=1\text{ for all
      }i\in[0,\alpha_t-1].\]
  \item[\eqref{eq:upper_bound_on}] If $y_t-y_{t-1}=1$, then by~\eqref{eq:lower_bound_on},
    $y_{t+i}=1$ for all $i\in[0,\alpha_t-1]$. For the sake of contradiction, assume $y_{t+i}=1$ for
    $i\in[\alpha_t,\beta_t]$. Then~\eqref{eq:switch_on} implies $z_{t+i}=0$ for all
    $i\in[1,\beta_t]$. By weak monotonicity
    $\left[s(\vect\beta,t+\beta_t),t+\beta_t\right]\subseteq[t+1,t+\beta_t]$, hence
    $y_{t+\beta_t}=0$ by~\eqref{eq:on-upper_bound}, which is the required contradiction.
  \item[\eqref{eq:lower_bound_off}] Suppose there is a pair $(t,i)$ with $t\in[0,n-1]$ and
    $i\in[0,\gamma_t-1]$ with $y_{t-1}=y_{t+i}=1$, $y_t=0$, and let $i$ be the smallest possible
    value. Then~(\ref{eq:switch_on}) implies $z_{t+i}=1$, and~(\ref{eq:off-lower_bound}) implies
    $y_{s(\gamma,t+i)-1}=0$. For $k=s(\vect\gamma,t+i)$, $[k,t+i]\supseteq[t,t+i]$. Let $l$ be the
    first period in $[k,t-1]$ with $y_l=1$. Then $z_l=1$ by~(\ref{eq:switch-on}), and
    \[\sum_{k\in[s(\vect\gamma,t+i),t+i]}z_{k}\geq z_l+z_{t+i}=2, \]
    which contradicts~(\ref{eq:off-lower_bound}).
  \item[\eqref{eq:upper_bound_off}] If $y_{t-1}-y_{t}=1$, then by~\eqref{eq:lower_bound_off},
    $y_{t+i}=0$ for all $i\in[0,\gamma_t-1]$. For the sake of contradiction, assume $y_{t+i}=0$ for
    $i\in[\gamma_t,\delta_t]$. Then~\eqref{eq:switch_on} implies $z_{t+i}=0$ for all
    $i\in[1,\delta_t]$. Set $k=s(\vect\delta,t+\delta_t)$. By weak monotonicity
    $\left[k,t+\delta_t\right]\subseteq[t+1,t+\delta_t]$, and~(\ref{eq:off-upper_bound}) implies
    $y_{k-1}=1$. This is the required contradiction because $k-1\in[t,t+\delta-1]$. \qedhere
  \end{description}   
\end{proof}
\subsection{Valid inequalities}\label{subsec:valid_inequalities}
In contrast to the non-cyclic situation studied in~\cite{QueyranneWolsey2017}, the polytopes $P$ and
$Q$ are not integral in general. In this subsection, let
$(\alpha_t,\beta_t,\gamma_t,\delta_t)=(\alpha,\beta,\gamma,\delta)$ for all
$t\in[0,n-1]$. Then~(\ref{eq:on-lower_bound}),~(\ref{eq:on-upper_bound}),~(\ref{eq:off-lower_bound}) and \eqref{eq:off-upper_bound} can be written as follows:
\begin{align}
 -y_t+ \sum_{i\in[0,\,\alpha-1]}z_{t-i} &\leq 0 &&t\in[0,n-1],\label{eq:lower_bound_on_a}\\
    y_t-\sum_{i\in[0,\,\beta-1]}z_{t-i} &\leq 0 &&t\in[0,n-1],\label{eq:upper_bound_on_a}\\
    y_t+ \sum_{i\in[1,\,\gamma]} z_{t+i}&\leq 1 &&t\in[0,n-1],\label{eq:lower_bound_off_a}\\
   -y_t- \sum_{i\in[1,\delta]} z_{t+i}&\leq -1 &&t\in[0,n-1].\label{eq:upper_bound_off_a}
\end{align}
Let $P_I$ denote the integer hull of $P$, that is $P_I=\conv(P\cap\ints^{2n})$.

\begin{proposition}\label{prop:bounds_on_z}  
  The inequalities
  \begin{align}
    \sum_{t\in[0,n-1]}z_t &\leq \left\lfloor n/(\alpha+\gamma)\right\rfloor,\label{eq:UB_z}\\
    \sum_{t\in[0,n-1]}z_t &\geq \left\lceil n/(\beta+\delta)\right\rceil\label{eq:LB_z}
  \end{align}
  are valid for $P_I$. If $\alpha<\beta$, $\gamma<\delta$ and $\lfloor
  n/(\alpha+\gamma)\rfloor>\lceil n/(\beta+\delta)\rceil$, then the following statements are true:
  \begin{enumerate}[(i)]
\item\label{item:facet_UB_z} If $n$ is not divisible by
  $\alpha+\gamma$, then $\dim P_I= 2n$ and~\eqref{eq:UB_z} is a facet of $P_I$.
\item\label{item:facet_LB_z} If $n$ is not divisible by
  $\beta+\delta$, then $\dim P_I= 2n$ and~\eqref{eq:LB_z} is a facet of $P_I$.
  \end{enumerate}  
\end{proposition}
\begin{proof}
  The upper bound~\eqref{eq:UB_z} comes from summing constraints~\eqref{eq:lower_bound_on_a}
  and~(\ref{eq:lower_bound_off_a}) over all $t$, and then using integrality
  to round the RHS. For the lower bound~\eqref{eq:LB_z} we do the same with
  constraints~(\ref{eq:upper_bound_on_a}) and~(\ref{eq:upper_bound_off_a}).
  
  In order to prove~(\ref{item:facet_UB_z}), we write $n=q(\alpha+\gamma)+r$ with
  $r\in[1,\,\alpha+\gamma-1]$, and set
  \[X = \left\{(\vect y,\vect z)\in P_I\,:\,z_0+z_1+\dots+z_{n-1}=q\right\}.\]
  The claim follows if we can show that $\dim X=2n-1$, or equivalently, the affine hull of $X$ is
  $\{(\vect y,\vect z)\,:\,z_0+\dots+z_{n-1}=q\}$. For this purpose, suppose $X$ lies in the affine
  subspace defined by
  \begin{equation}\label{eq:affine_relation}
    \sum_{t\in[0,n-1]}a_ty_t+\sum_{t\in[0,n-1]}b_tz_t=c.
  \end{equation}
  By assumption, there are vectors $(\vect y,\vect z),(\vect y',\vect z)\in X$, where
  \begin{align*}
    \vect y &=
              0\underbrace{00\dots0}_{d_1}\underbrace{11\dots1}_{d'_1}\underbrace{00\dots0}_{d_2}\underbrace{11\dots1}_{d'_2}\
              \dots\ \underbrace{00\dots0}_{d_q}\underbrace{11\dots1}_{d'_q},\\
    \vect y' &=
              1\underbrace{00\dots0}_{d_1}\underbrace{11\dots1}_{d'_1}\underbrace{00\dots0}_{d_2}\underbrace{11\dots1}_{d'_2}\
              \dots\ \underbrace{00\dots0}_{d_q}\underbrace{11\dots1}_{d'_q},\\
    \vect z &=
              \underbrace{00\dots0}_{d_1+1}1\underbrace{00\dots0}_{d'_1+d_2-1}1\ \dots\ \underbrace{00\dots0}_{d'_{q-1}+d_q-1}1\underbrace{00\dots0}_{d'_q-1}.
  \end{align*}
  Taking the difference between the two equations obtained from substituting $(\vect y,\vect z)$ and
  $(\vect y',\vect z)$ into~\eqref{eq:affine_relation}, we conclude $a_0=0$, and applying the same
  argument to the cyclic shifts of $(\vect y,\vect z)$ and $(\vect y',\vect z)$, $a_t=0$ for all
  $t\in[0,n-1]$. Applying a similar argument to vectors $(\vect y,\vect z),(\vect y',\vect z')\in X$
  with
  \begin{align*}
    \vect y &=
              0\underbrace{11\dots1}_{d_1}\underbrace{00\dots0}_{d'_1}\underbrace{11\dots1}_{d_2}\underbrace{00\dots0}_{d'_2}\
              \dots\ \underbrace{11\dots1}_{d_q}\underbrace{00\dots0}_{d'_q},\\
    \vect y' &=
              1\underbrace{11\dots1}_{d_1}\underbrace{00\dots0}_{d'_1}\underbrace{11\dots1}_{d_2}\underbrace{00\dots0}_{d'_2}\
              \dots\ \underbrace{11\dots1}_{d_q}\underbrace{00\dots0}_{d'_q},\\
    \vect z &=
              01\underbrace{00\dots0}_{d_1+d'_1-1}1\underbrace{00\dots0}_{d_2+d'_2-1}\ \dots\
              1\underbrace{00\dots0}_{d_q+d'_q-1},\\
    \vect z' &=
              10\underbrace{00\dots0}_{d_1+d'_1-1}1\underbrace{00\dots0}_{d_2+d'_2-1}\ \dots\ 1\underbrace{00\dots0}_{d_q+d'_q-1},
  \end{align*}
  and their cyclic shifts, we obtain $b_{t+1}-b_t=a_t=0$ for all $t\in[0,n-1]$. As a consequence,
  \eqref{eq:affine_relation} is a multiple of the relation $z_0+\dots+z_{n-1}=q$, and this
  concludes the proof of~(\ref{item:facet_UB_z}). The proof of~(\ref{item:facet_LB_z}) is similar.
 \end{proof}
\begin{example}
  Let $(\alpha,\beta,\gamma,\delta)=(1,2,1,2)$. For $n\in\{4,5\}$, $P_I$ is completely described
  by constraints~\eqref{eq:switch-on} through~\eqref{eq:variable-bounds}, together
  with~\eqref{eq:UB_z} and~\eqref{eq:LB_z}. Constraint~(\ref{eq:LB_z}) is a facet of $P_I$ unless
  $n$ is a multiple of $4$, and~(\ref{eq:UB_z}) is a facet whenever $n$ is odd.
\end{example}

\begin{proposition}\label{prop:bound_y}
  Let $n=q_1(\alpha+\delta)+r_1=q_2(\beta+\gamma)-r_2$ with $r_1\in[0,\,\alpha+\delta-1]$ and
  $r_2\in[0,\,\beta+\gamma-1]$. The inequalities
  \begin{align}
    \sum_{t\in[0,n-1]}y_t &\geq q_1\alpha+\min\{r_1,\alpha\}\label{eq:LB_y}\\
    \sum_{t\in[0,n-1]}y_t &\leq q_2\beta-\min\{r_2,\beta\}\label{eq:UB_y}
  \end{align}
 are valid for $P_I$.
\end{proposition}
\begin{proof}
Taking the sums of~\eqref{eq:lower_bound_on_a} through~\eqref{eq:upper_bound_off_a} over all $t\in[0,n-1]$, we obtain
\begin{align*}
  \alpha\sum_{t\in[0,n-1]}z_t\stackrel{\eqref{eq:lower_bound_on_a}}{\leq}\sum_{t\in[0,n-1]}y_t &\stackrel{\eqref{eq:upper_bound_on_a}}{\leq}\beta\sum_{t\in[0,n-1]}z_t, \\
  n-\delta\sum_{t\in[0,n-1]}z_t\stackrel{\eqref{eq:upper_bound_off_a}}{\leq}\sum_{t\in[0,n-1]}y_t &\stackrel{\eqref{eq:lower_bound_off_a}}{\leq} n-\gamma\sum_{t\in[0,n-1]}z_t.
\end{align*}
As a consequence,
\begin{align*}
  \sum_{t\in[0,n-1]}y_t &\geq
  \begin{cases}
    q_1\alpha+\alpha &\text{if }z_0+\dots+z_{n-1}\geq q_1+1,\\
    q_1\alpha+r_1&\text{if }z_0+\dots+z_{n-1}\leq q_1,
  \end{cases}\\
  \sum_{t\in[0,n-1]}y_t &\leq
  \begin{cases}
    q_2\beta-\beta &\text{if }z_0+\dots+z_{n-1}\leq q_2-1,\\
    q_2\beta-r_2&\text{if }z_0+\dots+z_{n-1}\geq q_2.
  \end{cases}\qedhere
\end{align*}
\end{proof}
\begin{example}
  For $n=7$ and $(\alpha,\beta,\gamma,\delta)=(1,2,1,2)$, (\ref{eq:LB_y}) and~(\ref{eq:UB_y}) become
  $\sum y_t\geq 3$ and $\sum y_t\leq 4$, respectively, and both of them are facets of $P_I$, as can
  be verified by hand or using software such as \texttt{polymake}\cite{Assarf2016}. 
\end{example}
\begin{example}
  For $(\alpha,\beta,\gamma,\delta)=(1,2,1,2)$ it can be checked that the following are valid
  inequalities for $P_I$:
    \begin{align}
  y_t+y_{t+1}-z_t+\sum_{i\in[2,n-2]}z_{t+i} &\leq \left\lfloor(n-1)/2\right\rfloor && t\in[0,n-1],  \label{eq:VI_1}\\
  y_t-y_{t-2}-z_t-z_{t-1}+\sum_{i\in[1,n-2]}z_{t+i}&\leq \left\lfloor(n-3)/2\right\rfloor &&
                                                                                                 t\in[0,n-1],\label{eq:VI_2}\\
  -y_t-y_{t+1}-z_{t+2}+\sum_{i\in[0,n-4]}z_{t-i} &\leq \left\lfloor(n-5)/2\right\rfloor && t\in[0,n-1].\label{eq:VI_3}  
    \end{align}
    For $n=6$, these are facets, and $P_I$ is completely described by~\eqref{eq:switch-on}
    through~\eqref{eq:variable-bounds}, \eqref{eq:LB_z},~\eqref{eq:VI_1},~\eqref{eq:VI_2},~\eqref{eq:VI_3}.
\end{example}

\subsection{An extended formulation}\label{subsec:extended_yz}
In Section~\ref{sec:flow} we defined the sets $T_0=\{\tau\in[0,n-1]\,:\,\tau+\beta_\tau\geq n\}$ and
$T_1=\{\tau\in[0,n-1]\,:\,\tau+\delta_\tau\geq n\}$ in order to classify the feasible solutions
according to the last switching period in the time horizon. More precisely, we have a partition
\[\Z=\bigcup_{i\in\{0,1\}}\bigcup_{\tau\in
    T_i}Z^{(i,\tau)}(n,\vect\alpha,\,\vect\beta,\vect\gamma,\,\vect \delta),\]
with
\begin{align*}
  Z^{(0,\tau)}(n,\vect\alpha,\,\vect\beta,\vect\gamma,\,\vect \delta) &= \Z\cap\left\{(\vect
                                                                        y,\,\vect
                                                                        z)\,:\,y_{\tau-1}=0,\,y_\tau=y_{\tau+1}=\dots=y_{n-1}=1\right\}\\
  Z^{(1,\tau)}(n,\vect\alpha,\,\vect\beta,\vect\gamma,\,\vect \delta) &= \Z\cap\left\{(\vect
                                                                        y,\,\vect
                                                                        z)\,:\,y_{\tau-1}=1,\,y_\tau=y_{\tau+1}=\dots=y_{n-1}=0\right\}
\end{align*}
We will describe the convex hulls of the sets $Z^{(i,\tau)}(n,\vect\alpha,\,\vect\beta,\vect\gamma,\,\vect \delta)$
following~\cite{QueyranneWolsey2017}, and then a result from disjunctive programming implies an extended
formulation for $\Z$. For the rest of this subsection we fix
$(n,\vect\alpha,\,\vect\beta,\vect\gamma,\,\vect \delta)$ and omit them from the notation, writing
for instance $Z$ instead of $\Z$. In order to describe the convex hulls of the sets
$Z^{(i,\tau)}$, we need additional parameters. The underlying idea is that the elements of a set
$Z^{(i,\tau)}$ correspond to on-off-sequences in the non-cyclic setting with known initial state as
described in \cite[Section 2]{QueyranneWolsey2017}. In order to capture the initial state, which is
determined by the pair $(i,\tau)$, we introduce an additional time period $-1$, which is essentially
a copy of period $n-1$. For instance, an element of $Z^{(0,n-3)}$ with $\alpha_{n-3}=5$ and
$\beta_{n-3}=8$, corresponds to a non-cyclic sequence with an on-switch in period $-1$, starting
with at least $3$ and at most $6$ on-periods. This is enforced by setting $\alpha_{-1}^{(0,n-3)}=3$
and $\beta^{(0,n-3)}=6$. In general, we introduce the following parameters:
\begin{align*}
  \alpha_{-1}^{(0,\tau)} &= \max\{1,\tau+\alpha_\tau-(n-1)\} &&\tau\in T_0,\\
  \beta_{-1}^{(0,\tau)} &= \tau+\beta_\tau-(n-1) &&\tau\in T_0,\\
  \left(\gamma_{-1}^{(0,\tau)},\,\delta_{-1}^{(0,\tau)}\right) &=
                                                                 \left(\gamma_{n-1},\,\delta_{n-1}\right)&&\tau\in
                                                                                                            T_0,\\
  \left(\alpha_{-1}^{(1,\tau)},\,\beta_{-1}^{(1,\tau)}\right) &=
                                                                 \left(\alpha_{n-1},\,\beta_{n-1}\right)&&\tau\in
                                                                                                           T_1,\\
  \gamma_{-1}^{(1,\tau)} &= \max\{1,\tau+\gamma_\tau-(n-1)\} &&\tau\in T_1,\\
  \delta_{-1}^{(1,\tau)} &= \tau+\delta_\tau-(n-1) &&\tau\in T_1.  
\end{align*}

For $t\in[0,n-1]$, we set $\eps_t^{(i,\tau)}=\eps_t$ for all $i\in\{0,1\}$, $\tau\in T_i$. Every
vector
$\vect\eps\in\{\vect\alpha^{(i,\tau)},\vect\beta^{(i,\tau)},\vect\gamma^{(i,\tau)},\vect\delta^{(i,\tau)}\,:\,i\in\{0,1\},\,
\tau\in T_i\}$ satisfies weak monotonicity, that is, $\eps_{t+1}\geq\eps_t-1$ for all $t\in[-1,n-2]$,
and therefore we can apply the results of~\cite[Section 3.1]{QueyranneWolsey2017}. We set
\[s'(\vect\eps,t)=\min\left\{k\in[-1,t]\,:\,k+\eps_k\geq t+1\right\}\]
and define polytopes $\tilde P^{(i,\tau)}\subseteq\reals^{2n+3}$ for $i\in\{0,1\}$, $\tau\in T_i$ by the
following constraints:
\begin{align*}
  z_{-1} &=y_{-1},\\
  z_t &\geq y_t-y_{t-1} &&t\in[0,n-1], \\
  \sum_{k\in[s'(\vect\alpha^{(i,\tau)},t),t]}z_k &\leq y_t &&t\in[0,n-1], \\
  y_t&\leq\sum_{k\in[s'(\vect\beta^{(i,\tau)},t),t]}z_k   &&t\in[0,n-1], \\
  \sum_{k\in[s'(\vect\gamma^{(i,\tau)},t),t]}z_k &\leq \lambda-y_{s'(\vect\gamma^{(i,\tau)},t)-1}
      &&t\in[0,n-1]:\,s'(\vect\gamma^{(i,\tau)},t)\geq 0, \\
  \lambda-y_{s'(\vect\delta^{(i,\tau)},t)-1} &\leq \sum_{k\in[s'(\vect\delta^{(i,\tau)},t),t]}z_k 
      &&t\in[0,n-1]:\,s'(\vect\delta^{(i,\tau)},t)\geq 0, \\
  0\leq y_t,\,z_t\leq \lambda&\leq 1&&t\in[0,n-1].
\end{align*}
The polytopes $\hat P^{(i,\tau)}\subseteq\reals^{2n+3}$ are defined as follows:
\begin{align*}
  \hat P^{(0,\tau)} &= \tilde P^{(0,\tau)}\cap\{(\vect y,\,\vect z,\,\lambda)\,:\,y_{\tau-1}=0,\,y_\tau=\dots=y_{n-1}=y_{-1}=\lambda\},\\
  \hat P^{(1,\tau)} &= \tilde P^{(1,\tau)}\cap\{(\vect y,\,\vect z,\,\lambda)\,:\,y_{\tau-1}=\lambda,\,y_\tau=\dots=y_{n-1}=y_{-1}=0\}.
\end{align*}
For $\lambda^*\in\reals$, let $\tilde P^{(i,\tau)}(\lambda^*)$ and $\hat P^{(i,\tau)}(\lambda^*)$ be the
slices of $\tilde P^{(i,\tau)}$ and $\hat P^{(i,\tau)}$, respectively, obtained by fixing $\lambda=\lambda^*$.
\begin{lemma}\label{lem:partition_of_Z}
  For every $i\in\{0,1\}$ and $\tau\in T_i$,
  $\conv(Z^{(i,\tau)})=f(\hat P^{(i,\tau)}(1))$, where $f:\reals^{2n+3}\to\reals^{2n}$ is the
  projection $(y_{-1},\,y_0,\dots,\,y_{n-1},\,z_{-1},\,z_0,\dots,\,z_{n-1},\,\lambda)\mapsto(y_{0},\,y_1,\dots,\,y_{n-1},\,z_{0},\,z_1,\dots,\,z_{n-1})$.
\end{lemma}
\begin{proof}
  The polytopes $\tilde P^{(i,\tau)}(1)$ are integral by~\cite[Theorem 2]{QueyranneWolsey2017}, and
  since fixing some binary variables does not destroy integrality, the polytopes $\hat P^{(i,\tau)}(1)$ are integral. The
  result follows since $f(\hat P^{(i,\tau)}(1))$ is a formulation for $Z^{(i,\tau)}$.
\end{proof}
We can now apply a result from disjunctive programming (see \cite{Jeroslow1984,Balas1985,Balas1998})
to obtain an extended formulation.
\begin{proposition}\label{prop:disjunctive_formulation}
  The polytope $\hat P\subseteq\reals^{2n+(2n+3)(\lvert T_0\rvert+\lvert T_1\rvert)}$ defined by the
  constraints
  \begin{align*}
    \sum_{i\in\{0,1\}}\sum_{\tau\in T_i}\lambda^{(i,\tau)} &= 1,\\
    \left(\vect y^{(i,\tau)},\,\vect z^{(i,\tau)},\,\lambda^{(i,\tau)}\right) &\in \hat P^{(i,\tau)} &&
                                                                                                   i\in\{0,1\},\,\tau\in T_i,\\
    \sum_{i\in\{0,1\}}\sum_{\tau\in T_i}y_t^{(i,\tau)}   &= y_t && t\in[0,n-1],\\
    \sum_{i\in\{0,1\}}\sum_{\tau\in T_i}z_t^{(i,\tau)}   &= z_t && t\in[0,n-1].
  \end{align*}
  provides an extended formulation for $Z$.  
\end{proposition}
\section{Open problems}\label{sec:conclusion}
We conclude with some open problems. Trying to proceed along the lines of~\cite{QueyranneWolsey2017}, it is natural to consider the following
two problems.
\begin{problem}
  Characterize the convex hull of $\Z$ in terms of the $x$-variables, that is, determine the integer hull of the polytope $Q$.
\end{problem}
\begin{problem}
  Characterize the convex hull of $\Z$ in terms of the original $y$- and $z$-variables, that is,
  determine the integer hull of the polytope $P$. 
\end{problem}
In particular, in both cases we would like to know if the number of facets is polynomial. In the
small cases we have analyzed with \texttt{polymake}~\cite{Assarf2016} we observed that $\proj_{y,z}(Q)=P$. This motivates the
following question.
\begin{problem}
  Is it true that $\proj_{y,z}(Q)=P$ in general?
\end{problem}

\bigskip

\subsection*{Acknowledgment}
Thomas Kalinowski and Hamish Waterer are supported by the Australian Research Council and Aurizon
Network Pty Ltd under the grant LP140101000. The work of Tomas Lid\'en is performed as part of the
research project ``Efficient planning of railway infrastructure maintenance'', funded by the Swedish
Transport Administration with the grant TRV 2013/55886 and conducted within the national research
program ``Capacity in the Railway Traffic System''.

\printbibliography

\end{document}